\documentclass[12pt]{article}
\usepackage{amssymb}
\usepackage{mathrsfs}
\usepackage{amsmath}
\usepackage[dvips]{graphicx}
\usepackage{latexsym,bm}
\usepackage{amsmath,amssymb,amsfonts}
\usepackage{indentfirst}
\usepackage{amsthm}
\usepackage{cases}

\usepackage{pifont}

\allowdisplaybreaks
\author{Jun-Ming Zhu\\
Department of Mathematics, East China Normal University,\\ Dongchuan
Road 500, Shanghai 200241, China\\
\\E-mail: junming\_zhu@163.com \\
}
\title{The number of convex pentagons and hexagons in an
$n$-triangular net
}
\date{}
\newtheorem{lem}{\quad\textbf{\Large Lemma}}[section]
\newtheorem{thm}[lem]{\quad\textbf{\Large Theorem}}

\begin{document}
 \maketitle
 \setcounter{section}{0}
\begin{abstract}
In this paper, we obtain the counting formulaes of convex pentagons
and convex hexagons, respectively, in an $n$-triangular net by
solving the corresponding recursive formulaes.
 \end{abstract}
  \textbf{Key words: } $n$-triangular net, convex pentagon, convex
hexagon, regular hexagon
      \textbf{AMS 2000: } 05A15, 39A10
\section{ Introduction }
In the paper \cite{zhu}, the author obtained the counting formulaes
of triangles and quadrilaterals in an $n$-triangular net,
respectively,  by induction. In this paper, we get the counting
formulaes of convex pentagons and hexagons in an $n$-triangular net,
respectively, by solving two corresponding recursive formulaes. We
first give the following definition.

  \textbf{Definition 1.1} Divide  each edge of a (regular) triangular
  into $n$ $(n\geq1)$ equal
  parts, and then construct $n-1$ segments between the dividing points on
  two edges parallel to the third one. Then the graph we get is
  called an $n$-(regular) triangular net.

  See fig. 1,  fig. 2  and fig. 3  in the following.

If $n=1$, then the $n$-triangular net reduces to a triangular.

  By the property of affine transformation, we know that the numbers
 of convex pentagons and convex hexagons in an
$n$-triangular net are only dependent on
$n$
 but independent from the
shape and the size
  of the triangular.

%\newpage
\begin{picture}(120,220)
   \setlength{\unitlength}{1mm}

\put(15,57)
{\begin{picture}(40,30)
\newcounter{tqot}
\setcounter{tqot}{0}
\newcounter{vqot}
\setcounter{vqot}{0}
\newcounter{trot}
\setcounter{trot}{0}
\newcommand{\swrt}[1]{\makebox(0,0)[tr] {$P_{#1}$}}
\newcommand{\szwrt}[1]{\makebox(0,0)[br] {$A_{#1}$}}
\newcommand{\sjwrt}[1]{\makebox(0,0)[bl] {$B_{#1}$}}

  \setlength{\unitlength}{1mm}
   \put(0,0){\line(1,0){30}}
   \put(0,0){\line(1,2){15}}
   \put(30,0){\line(-1,2){15}}
   \put(5,10){\line(1,0){20}}
   \put(10,20){\line(1,0){10}}
   \put(10,0){\line(-1,2){5}}
   \put(20,0){\line(-1,2){10}}
   \put(10,0){\line(1,2){10}}
   \put(20,0){\line(1,2){5}}

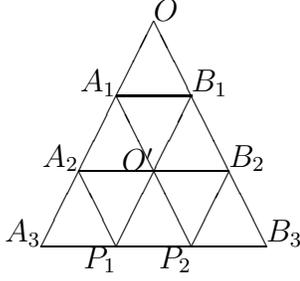
   \put(15,-5){\makebox(0,0)[t]{fig. 1: $3$-triangular net $OA_3B_3$}}
    \put(15,30){\makebox(0,0)[bl]{$O$}}
    \put(15,10){\makebox(0,0)[br]{$O'$}}
     \multiput(10,20)(-5,-10){3}
    {\addtocounter{tqot}{1}\szwrt{\arabic{tqot}}}

    \multiput(20,20)(5,-10){3}
    {\addtocounter{trot}{1}\sjwrt{\arabic{trot}}}
     \multiput(10,0)(10,0){2}
    {\addtocounter{vqot}{1}\swrt{\arabic{vqot}}}
\end{picture}}

\put(0,0)
{\begin{picture}(40,40)
\setlength{\unitlength}{0.7mm}
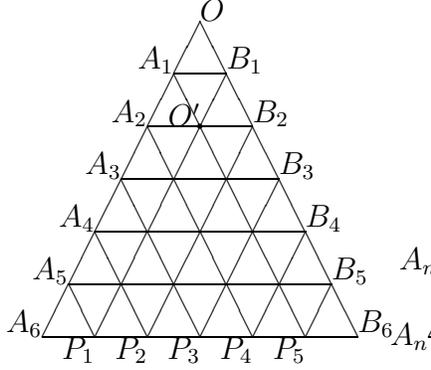
 \put(30,-8){\makebox(0,0)[t]{fig. 2: $6$-triangular net $OA_6B_6$}}

\newcounter{pot}
\setcounter{pot}{0}

\newcounter{sot}
\setcounter{sot}{0}

\newcounter{jsh}
\setcounter{jsh}{0}
\newcommand{\wrt}[1]{\makebox(0,0)[tr] {$P_{#1}$}}
\newcommand{\zwrt}[1]{\makebox(0,0)[tr] {$A_{#1}$}}
\newcommand{\jwrt}[1]{\makebox(0,0)[tl] {$B_{#1}$}}
  \put(30,40){\makebox(0,0)[br]{$O'$}}
\put(0,0){\line(1,0){60}}
 \put(0,0){\line(1,2){30}}

   \put(5,10){\line(1,0){50}}
   \put(10,20){\line(1,0){40}}
   \put(15,30){\line(1,0){30}}
   \put(20,40){\line(1,0){20}}
    \put(25,50){\line(1,0){10}}

   \put(10,0){\line(-1,2){5}}
   \put(20,0){\line(-1,2){10}}
   \put(30,0){\line(-1,2){15}}
   \put(40,0){\line(-1,2){20}}
   \put(50,0){\line(-1,2){25}}
   \put(60,0){\line(-1,2){30}}

   \put(10,0){\line(1,2){25}}
   \put(20,0){\line(1,2){20}}
   \put(30,0){\line(1,2){15}}
   \put(40,0){\line(1,2){10}}
   \put(50,0){\line(1,2){5}}
   \put(30,40){\circle*{0.8}}
   \put(30,60){\makebox(0,0)[bl]{$O$}}
    \multiput(25,55)(-5,-10){6}
    {\addtocounter{pot}{1}\zwrt{\arabic{pot}}}

    \multiput(35,55)(5,-10){6}
    {\addtocounter{sot}{1}\jwrt{\arabic{sot}}}
    \multiput(10,0)(10,0){5}
    {\addtocounter{jsh}{1}\wrt{\arabic{jsh}}}

\end{picture}}

\put(50,0){
\begin{picture}(90,100)

\newcounter{qot}
\setcounter{qot}{0}

\newcounter{rot}
\setcounter{rot}{0}

\setlength{\unitlength}{1mm}
\newcommand{\nzwrt}[1]{\makebox(0,0)[tr] {$A_{#1}$}}
\newcommand{\njwrt}[1]{\makebox(0,0)[tl] {$B_{#1}$}}
 \put(0,0){\line(1,0){90}}
 \put(15,30){\line(1,2){30}}
\put(75,30){\line(-1,2){30}}
 \multiput(5,10)(2,0){40}{\line(1,0){1}}
 \multiput(10,20)(2,0){35}{\line(1,0){1}}
 \multiput(15,30)(2,0){30}{\line(1,0){1}}
 \multiput(5,10)(10,0){7}{\multiput(0,0)(2.5,5){4}{\line(1,2){1.8}}}
 \multiput(25,10)(10,0){7}{\multiput(0,0)(-2.5,5){4}{\line(-1,2){1.8}}}
  \multiput(75,10)(3,6){2}{\line(1,2){1.8}}
  \multiput(15,10)(-3,6){2}{\line(-1,2){1.8}}

   \put(40,70){\makebox(0,0)[bl]{$O'$}}

   \put(20,40){\line(1,0){50}}
    \put(25,50){\line(1,0){40}}
    \put(30,60){\line(1,0){30}}
  \put(35,70){\line(1,0){20}}
  \put(40,80){\line(1,0){10}}

   \multiput(10,0)(10,0){9}{\line(-1,2){5}}

   \put(25,30){\line(-1,2){5}}
    \put(35,30){\line(-1,2){10}}
    \put(45,30){\line(-1,2){15}}
    \put(55,30){\line(-1,2){20}}
   \put(65,30){\line(-1,2){25}}

   \put(25,30){\line(1,2){25}}
   \put(35,30){\line(1,2){20}}
   \put(45,30){\line(1,2){15}}
    \put(55,30){\line(1,2){10}}
    \put(65,30){\line(1,2){5}}

     \multiput(0,0)(10,0){9}{\line(1,2){5}}
     \put(45,90){\makebox(0,0)[bl]{$O$}}

    \put(0,0){\makebox(0,0)[dr]{$A_n$}}
    \put(5,10){\makebox(0,0)[dr]{$A_{n-1}$}}
     \put(10,0){\makebox(0,0)[tl]{$P_1$}}
     \put(90,0){\makebox(0,0)[bl]{$B_n$}}
      \put(85,10){\makebox(0,0)[bl]{$B_{n-1}$}}

     \put(80,0){\makebox(0,0)[tl]{$P_{n-1}$}}

     \multiput(40,85)(-5,-10){6}
    {\addtocounter{qot}{1}\nzwrt{\arabic{qot}}}

    \multiput(50,85)(5,-10){6}
    {\addtocounter{rot}{1}\njwrt{\arabic{rot}}}
      \put(45,70){\circle*{0.8}}
    \put(45,-5){\makebox(0,0)[t]{fig. 3}}
\end{picture}}
\end{picture}
   \vspace{1.2cm}

\section{The counting formulae of convex pentagons
in an $n$-triangular net}
  \begin{thm} The number $P(n)$ of convex pentagons in an
$n$-triangular net is
  \begin{equation}P(n)=
\begin{cases} \frac
{1}{10}(12k^5+25k^4+5k^3-10k^2-2k),  \
&n=2k+1 \ (k=0,1,2,\cdots ),\\
 \frac{1}{10}(12k^5-5k^4-15k^3+5k^2+3k),
 \  &n=2k \ (k=1,2,\cdots ).
\end{cases}    \label{pen}
\end{equation}

  \end{thm}

  \begin{proof}
 Without loss of generosity,
   we suppose that the divided triangular is regular.
 Then there is only one acute interior angle in each convex
 pentagon.

 Observing the figures above, we know that $P(n)$ can be expressed as
 $2P(n-1)$ (all the pentagons in the $(n-1)$-triangular net $A_1A_nP_{n-1}$
 or $B_1P_1B_{n}$),
 subtracting $P(n-2)$ (pentagons in the $(n-2)$-triangular net $O'P_1P_{n-1}$
 ),
 and then adding the number,
  denoted by $f(n)$,
  of the pentagons of which there are vertexes both on $OA_n$
 and on $OB_n$. This is
 \begin{equation}
 P(n)=2P(n-1)-P(n-2)+f(n).  \label{guanjian}
 \end{equation}

Obviously, we have $P(1)=0$, $P(2)=0$, $P(3)=3$ and $f(1)=0$,
$f(2)=0$, $f(3)=3$.

Now it is crucial to get the expression of $f(n)$. Note that, from
the definition of $f(n)$, $f(n)-f(n-1)$ is the number of pentagons
of which there are vertexes on $OA_n$, $OB_n$ and $A_nB_n$ and of
which
 the unique acute vertex must be on $OA_n$, $OB_n$ or $A_nB_n$.

The number of all the pentagons of which there are vertexes on
$OA_n$, $OB_n$ and $A_nB_n$ and  of which the acute vertex is one of
$O,~A_n$ and $B_n$ is
 $$
 (1+2+3+\cdots+(n-2))\times3.
 $$
The number of all the pentagons of which there are vertexes on
$OA_n$, $OB_n$ and $A_nB_n$ and  of which the acute vertex is one of
 $A_1,\ A_2,\ \cdots, \ A_{n-1}, B_1,\ B_2,\ \cdots ,\ B_{n-1}
 $ and $\ P_1,$ $ P_2,$ $ \cdots,$  $P_{n-1}$
 is
$$
1+2+\cdots+(k-1)+(k-1)+\cdots+2+1 ~~(\mbox{if}~~n=2k),
$$
or
 $$
 1+2+\cdots+(k-2)+(k-1)+(k-2)+\cdots+2+1 £¬~~(\mbox{if}~~n=2k+1).
$$
So we have
    \begin{eqnarray*}
    f(2k)-f(2k-1)&=&(1+2+3+\cdots+(2k-2))\times3\\
&&+1+2+\cdots+(k-1)+(k-1)+\cdots+2+1
\end{eqnarray*}
and
 \begin{eqnarray*}f(2k+1)-f(2k)&=&(1+2+3+\cdots+(2k-1))\times3
\\&&+1+2+\cdots+(k-2)+(k-1)+(k-2)+\cdots+2+1\\
&&\hspace{7.5cm}.
\end{eqnarray*}
This is
\begin{eqnarray*}
 &&f(2k) = f(2k-1)+3(3k^2-5k+2),\\
 && f(2k+1)
=f(2k)+3(3k^2-2k).
\end{eqnarray*}
 Iterating the above formulaes gives
 \begin{eqnarray*}
f(2)&=&f(1)+3(3\cdot1^2-5\cdot1+2), \\
f(3)&=&f(2)+3(3\cdot1^2-2\cdot1),\\
f(4)&=&f(3)+3(3\cdot2^2-5\cdot2+2), \\
f(5)&=&f(4)+3(3\cdot2^2-2\cdot2),\\
&&\cdots\cdots\cdots\cdots,\\
 f(2k)&=&f(2k-1)+3(3\cdot k^2-5\cdot k+2), \\
f(2k+1)&=&f(2k)+3(3\cdot k^2-2\cdot k).
\end{eqnarray*}
Overlay the above formulaes and note that $f(1)=0$ to get
\begin{eqnarray*}
&f(2k+1)=\frac{3}{2}(4k^3-k^2-k),\\
&f(2k)=\frac{3}{2}(4k^3-7k^2+3k).
\end{eqnarray*}
From (\ref{guanjian}), we have $P(n)-P(n-1)=P(n-1)-P(n-2)+f(n)$.
Then
  \begin{eqnarray*}
  &P(2k)-P(2k-1)=P(2k-1)-P(2k-2)+f(2k),\ (k=1,2,\cdots\cdots),\\
&P(2k+1)-P(2k)=P(2k)-P(2k-1)+f(2k+1),\ (k=0,1,2,\cdots\cdots).
\end{eqnarray*}
Overlaying again, we have
\begin{eqnarray*}
P(2k+1)-P(2k)&=&3k^4+2k^3-\frac{3}{2}k^2-\frac{1}{2}k,\\
P(2k)-P(2k-1)&=&3k^4-4k^3+k .
\end{eqnarray*}
Overlaying a third time, we get
 \begin{eqnarray*}
 P(2k+1)=\frac{6}{5}k^5+\frac{5}{2}k^4+\frac{1}{2}k^3-k^2-\frac{1}{5}k,\\
P(2k)=\frac{6}{5}k^5-\frac{1}{2}k^4-
\frac{3}{2}k^3+\frac{1}{2}k^2+\frac{3}{10}k,
\end{eqnarray*}
which completes the proof.

  \end{proof}
The difference equation (\ref{guanjian}) is fundamental in this
paper. This is a linear recurrence relation of order $2$ and
 can also be solved using the method usually used in solving
  recursive formulaes (see, for
 example, \cite[p.218--234, \S 7.2--7.3]{brualdi}). Our method is different from that of \cite{zhu}. Obviously,
  formulaes of this form can also
 be used to get the counting formulaes of
 triangles and quadrilaterals in \cite{zhu} and seem to be
 more understandable.
 We will
also use the equation of this form to get number of convex hexagons
in an $n$-triangular net  in the following.

\section{The counting formulae of convex hexagons
in an $n$-triangular net}
  \begin{thm}
  The number $H(n)$ of convex hexagons in an
$n$-triangular net is
\begin{equation}   \label{hexagon}
    H(n)=\begin{cases}\frac{1}{60}
(8k^6+24k^5+25k^4+10k^3-3k^2-4k),&n=2k+1 (k=0,1,2,\cdots),
\\\frac{1}{60}(8k^6-5k^4-3k^2), &n=2k
(k=1,2,\cdots).
\end{cases}
\end{equation}
  \end{thm}
\begin{proof}
  Just as the analysis in the proof of theorem 2.1, we have
   \begin{equation}
   H(n)=2H(n-1)-H(n-2)+g(n), \label{leisi}
   \end{equation}
 where $g(n)$ denote the number
of the hexagons of which there are vertexes both on $OA_n$
 and on $OB_n$.
Obviously, we have $H(1)=0$, $H(2)=0$, $H(3)=1$ and $g(1)=0$,
$g(2)=0$, $g(3)=1$.

  The equation (\ref{leisi}) is similar to the equation (\ref{guanjian}).
   So we solve it in the same way as (\ref{guanjian}).
 We have
  \begin{eqnarray*}
  g(2k+1)&=&g(2k)\\&&+1+2+3+4+\cdots+(2k-4)+(2k-3)+(2k-2)+(2k-1)
\\&&+2+3+4+\cdots+(2k-4)+(2k-3)+(2k-2)+(2k-2)
\\&&+3+4+\cdots+(2k-4)+(2k-3)+(2k-3)+(2k-3)
\\&&+\cdots\cdots\cdots\cdots
\\&&\underbrace{+(k-2)+(k-1)+k+(k+1)+(k+2)+(k+2)+\cdots+(k+2)}
\limits_{\hbox{The number is }k+2}
\\&&\underbrace{+(k-1)+k+(k+1)+(k+1)+(k+1)+\cdots+(k+1)}\limits_{{\hbox{
The number is }}k+1}
\\&&\underbrace{+k+k+k+k+\cdots+k}\limits_{{\hbox{
The number is }}k}
\\&&+\cdots\cdots\cdots\cdots
\\&&+2+2
\\&&+1
\\&=&g(2k)+1^2+2^2+\cdots+(2k-1)^2
\\&&-(+1+2+3+4+\cdots+(2k-4)+(2k-3)+(2k-2)
\\&&+1+2+3+4+\cdots+(2k-4)
\\&&+\cdots\cdots\cdots\cdots
\\&&+1+2+3+4\\&&+1+2)
\\&=&g(2k)+{k\over2}(4k^2-3k+1),
\end{eqnarray*}
and
 \begin{eqnarray*}g(2k)&=&g(2k-1)\\&&+1+2+3+4+\cdots+(2k-5)+(2k-4)+(2k-3)+(2k-2)
\\&&+2+3+4+\cdots+(2k-5)+(2k-4)+(2k-3)+(2k-3)
\\&&+3+4+\cdots+(2k-5)+(2k-4)+(2k-4)+(2k-4)
\\&&+\cdots\cdots\cdots\cdots
\\&&\underbrace{+(k-2)+(k-1)+k+(k+1)+(k+1)+\cdots+(k+1)}\limits_{{\hbox{
The number is }}k+1}
\\&&\underbrace{+(k-1)+k+k+k+k+\cdots+k}\limits_{{\hbox{
The number is }}k}\\
&&\underbrace{+(k-1)+(k-1)+(k-1)+\cdots+(k-1)}\limits_{{\hbox{ The
number is }}k-1}
\\&&+\cdots\cdots\cdots\cdots
\\&&+2+2
\\&&+1
\\&=&g(2k-1)+1^2+2^2+\cdots+(2k-2)^2
\\&&-(+1+2+3+4+\cdots+(2k-5)+(2k-4)+(2k-3)
\\&&+1+2+3+4+\cdots+(2k-5)
\\&&+\cdots\cdots\cdots\cdots
\\&&+1+2+3\\&&+1)
\\&=&g(2k-1)+{k-1\over2}(4k^2-5k+2).
\end{eqnarray*}
So we have
 \begin{eqnarray*}
 g(2k)&=&{1\over2}k(k-1)(2k^2-2k+1),\\
 g(2k+1)&=&k^4,
\end{eqnarray*}
 and then
 \begin{eqnarray*}
&H(2k)-H(2k-1)={k\over30}(12k^4-15k^3+5k^2-2)
\\&H(2k+1)-H(2k)={k\over30}(12k^4+15k^3+5k^2-2).
\end{eqnarray*}
By overlaying, we get (\ref{hexagon}). This completes the proof.
\end{proof}

%
%  Denote the number of regular hexagons in an
%$n$-regular triangular net by $R(n)$. For convenience, we define
%$R(0)=0$. Then we have
% \begin{thm}
%  The number $R(n)$ of regular hexagons in an
%$n$-regular triangular net is
%\begin{equation}
%\mathrm{R}(n)=\begin{cases}{1\over2}(3k^3-k)&n=3k
%\\{3\over2}(k^3+k^2)&n=3k+1\qquad\qquad k=0,1,2,\cdots\cdots
%\\{3\over2}(k^3+2k^2+k)&n=3k+2\end{cases}
%\end{equation}
%\end{thm}
%   This theorem can be proved easily by induction. And so the proof is
%   omitted.
 %\section{Corollary}
%    From theorem 2.1 and 3.1, we can easily get the following
% corollary. We write it as theorem 4.1.
%   \begin{thm} When $k=0,1,2,\cdots\cdots$,
%    the numbers
%\begin{eqnarray*}
%&\frac{6}{5}k^5+\frac{5}{2}k^4+\frac{1}{2}k^3-k^2-\frac{1}{5}k,
%&\tfrac{6}{5}k^5-\tfrac{1}{2}k^4-\tfrac{3}{2}k^3
%+\tfrac{1}{2}k^2+\tfrac{3}{10}k,\\
%&\frac{1}{60}(8k^6+24k^5+25k^4+10k^3-3k^2-4k),
%&\tfrac{1}{60}(8k^6-5k^4-3k^2)
%\end{eqnarray*}
% are all integers.
%
%   \end{thm}

 \end{document}